\documentclass[12pt,a4paper]{amsart}
\usepackage{amsmath,amssymb,color,multicol,setspace}
\usepackage{hyperref}
\usepackage[utf8,utf8x]{inputenc}
\usepackage{fullpage}

\usepackage{longtable}
\usepackage{graphicx}
\usepackage{pict2e}
\usepackage{xy,amscd}
\usepackage{epsfig}
\usepackage{rotating}
\usepackage{xspace}
\xyoption{all}
\usepackage{color}

\usepackage{prelim2e}

\usepackage{enumitem}

\newtheorem{propo}{Proposition}[section]
\newtheorem{corol}[propo]{Corollary}
\newtheorem{theor}[propo]{Theorem}
\newtheorem{lemma}[propo]{Lemma}

\theoremstyle{definition}
\newtheorem{defin}[propo]{Definition}
\newtheorem{examp}[propo]{Example}

\theoremstyle{remark}
\newtheorem{remar}[propo]{Remark}

\newcommand{\NN }{\mathbb{N}}

\newcommand{\RR }{\mathbb{R}}

\newcommand{\QQ }{\mathbb{Q}}
\newcommand{\ZZ }{\mathbb{Z}}

\newcommand{\id }{\mathrm{id}}

\DeclareMathOperator{\Aut}{Aut}

\DeclareMathOperator{\GL}{GL}

\DeclareMathOperator{\rank}{rank}

\newcommand{\ndZ }{\mathbb{Z}}

\newcommand{\al }{\alpha }

\DeclareMathOperator{\Vol}{Vol}

\DeclareMathOperator{\Mor}{Mor}
\DeclareMathOperator{\Obj}{Obj}

\DeclareMathOperator{\md}{mod}

\newcommand{\Ic }{\mathcal I}

\newcommand{\RS}{\mathcal{R}}
\newcommand{\Ss}{\mathcal{S}}

\newcommand{\Ac }{\mathcal{A}}

\newcommand{\Kc }{\mathcal{K}}

\newcommand{\Wc }{\mathcal{W}}
\newcommand{\Cc }{\mathcal{C}}

\newcommand{\cEs }{\mathcal{F}}

\newcommand{\rfl }{\rho }
\newcommand{\s }{\sigma }

\newcommand{\coord}{\Upsilon}

\definecolor{darkgreen}{rgb}{0.0,0.1,0.6}
\newcommand{\df}[1]{{\bf\color{darkgreen} #1}}

\title[A bound for crystallographic arrangements]
{A bound for crystallographic arrangements}

\author{Michael~Cuntz}
\address{Michael Cuntz, Leibniz Universit\"at Hannover,
Institut f\"ur Algebra, Zah\-lentheorie und Diskrete Mathematik,
Fakult\"at f\"ur Mathematik und Physik,
Wel\-fengarten 1,
D-30167 Hannover, Germany}
\email{cuntz@math.uni-hannover.de}

\begin{document}


\begin{abstract}
A crystallographic arrangement is a set of linear hyperplanes satisfying a certain integrality property and decomposing the space into simplicial cones. Crystallographic arrangements were completely classified in a series of papers by Heckenberger and the author. However, this classification is based on two computer proofs checking millions of cases. In the present paper, we prove without using a computer that, up to equivalence, there are only finitely many irreducible crystallographic arrangements in each rank greater than two.
\end{abstract}

\maketitle

\section{Introduction}

A simplicial arrangement is a finite set of linear hyperplanes in a real vector space which decomposes its complement into open simplicial cones, cf.\ \cite{a-Melchi41}.
Simpliciality is thus the extreme case when every chamber in a real hyperplane arrangement has the smallest possible number of walls.
It is not surprising that some of the most prominent arrangements, as for example the real reflection arrangements are simplicial.
Apart from a catalogue by Gr\"unbaum \cite{p-G-09}, some more results such as \cite{p-C12}, and the seminal result by Deligne \cite{MR0422673}, not much is known about simplicial arrangements in general.
Another way to obtain nice results is to restrict to smaller classes of arrangements, as for example in \cite{CM17} where all supersolvable simplicial arrangements are classified, or to consider larger classes as for instance infinite arrangements with similar properties, see \cite{CMW}.

Motivated by certain quantum groups, Heckenberger and the author classified another smaller class, the so called crystallographic arrangements. These are simplicial arrangements in a lattice with an additional saturation property (see Definition \ref{cryst:arr}). For instance, Weyl arrangements are crystallographic.
Although this sounds very special, notice that it appears that the class of crystallographic arrangements is not so small compared to the class of all simplicial arrangements: almost all of the known rational simplicial arrangements are crystallographic, and the number of known non-rational sporadic simplicial arrangements is very small.

\begin{theor}[cf.\ {\cite{p-CH09a}, \cite{p-CH09b}, \cite{p-CH09c}, \cite{p-CH09d}, \cite{p-CH10}, \cite{p-C10}}]\label{cryarrclas}
There are (up to equivalence) exactly three families of irreducible crystallographic arrangements:
\begin{enumerate}
\item The family of rank two parametrized by triangulations of convex $n$-gons by non-intersecting diagonals.
\item For each rank $r>2$, arrangements of type $A_r$, $B_r$, $C_r$
and $D_r$, and a further series of $r-1$ arrangements.
\item Another $74$ ``sporadic'' arrangements of rank $r$, $3\le r \le 8$.
\end{enumerate}
\end{theor}

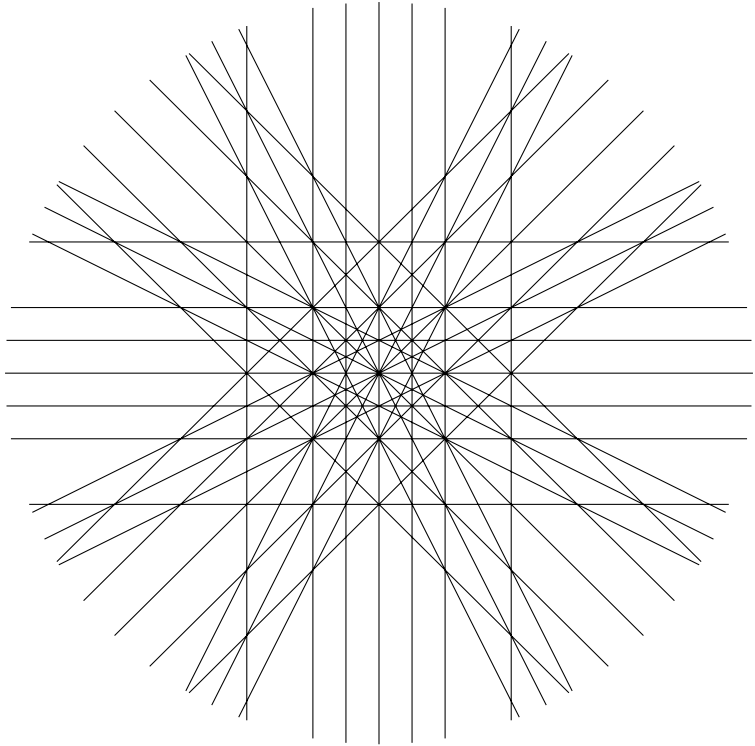
\begin{figure}
\begin{center}
\setlength{\unitlength}{0.7pt}
\begin{picture}(600,400)(0,200)
\moveto(500.000000000000000000000000000,400.000000000000000000000000000)\lineto(100.000000000000000000000000000,400.000000000000000000000000000)
\moveto(478.885438199983175712733893499,310.557280900008412143633053251)\lineto(121.114561800016824287266106501,489.442719099991587856366946749)
\moveto(441.421356237309504880168872421,258.578643762690495119831127579)\lineto(158.578643762690495119831127579,541.421356237309504880168872421)
\moveto(300.000000000000000000000000000,600.000000000000000000000000000)\lineto(300.000000000000000000000000000,200.000000000000000000000000000)
\moveto(441.421356237309504880168872421,541.421356237309504880168872421)\lineto(158.578643762690495119831127579,258.578643762690495119831127579)
\moveto(478.885438199983175712733893499,489.442719099991587856366946749)\lineto(121.114561800016824287266106501,310.557280900008412143633053251)
\moveto(389.442719099991587856366946749,221.114561800016824287266106501)\lineto(210.557280900008412143633053251,578.885438199983175712733893499)
\moveto(389.442719099991587856366946749,578.885438199983175712733893499)\lineto(210.557280900008412143633053251,221.114561800016824287266106501)
\moveto(403.304908124365989904268758083,571.254477189404603588495298060)\lineto(224.979363123095911071765016402,214.603387186864445923487814698)
\moveto(317.677669529663688110021109053,599.217218131365342082889509328)\lineto(317.677669529663688110021109053,200.782781868634657917110490672)
\moveto(403.304908124365989904268758083,228.745522810595396411504701939)\lineto(224.979363123095911071765016402,585.396612813135554076512185302)
\moveto(485.396612813135554076512185302,475.020636876904088928234983598)\lineto(128.745522810595396411504701939,296.695091875634010095731241917)
\moveto(457.989821533686490069411686515,522.634482474359113849369468410)\lineto(177.365517525640886150630531591,242.010178466313509930588313485)
\moveto(335.355339059327376220042218105,596.850196850295275492125861220)\lineto(335.355339059327376220042218105,203.149803149704724507874138780)
\moveto(457.989821533686490069411686515,277.365517525640886150630531590)\lineto(177.365517525640886150630531591,557.989821533686490069411686515)
\moveto(485.396612813135554076512185302,324.979363123095911071765016402)\lineto(128.745522810595396411504701939,503.304908124365989904268758083)
\moveto(499.217218131365342082889509328,382.322330470336311889978890947)\lineto(100.782781868634657917110490672,382.322330470336311889978890947)
\moveto(496.850196850295275492125861221,364.644660940672623779957781895)\lineto(103.149803149704724507874138780,364.644660940672623779957781895)
\moveto(496.850196850295275492125861221,435.355339059327376220042218105)\lineto(103.149803149704724507874138780,435.355339059327376220042218105)
\moveto(499.217218131365342082889509328,417.677669529663688110021109053)\lineto(100.782781868634657917110490672,417.677669529663688110021109053)
\moveto(471.254477189404603588495298060,296.695091875634010095731241917)\lineto(114.603387186864445923487814698,475.020636876904088928234983598)
\moveto(422.634482474359113849369468409,242.010178466313509930588313485)\lineto(142.010178466313509930588313485,522.634482474359113849369468410)
\moveto(264.644660940672623779957781895,596.850196850295275492125861220)\lineto(264.644660940672623779957781895,203.149803149704724507874138780)
\moveto(422.634482474359113849369468409,557.989821533686490069411686515)\lineto(142.010178466313509930588313485,277.365517525640886150630531590)
\moveto(471.254477189404603588495298060,503.304908124365989904268758083)\lineto(114.603387186864445923487814698,324.979363123095911071765016402)
\moveto(375.020636876904088928234983598,214.603387186864445923487814698)\lineto(196.695091875634010095731241917,571.254477189404603588495298060)
\moveto(282.322330470336311889978890947,599.217218131365342082889509328)\lineto(282.322330470336311889978890947,200.782781868634657917110490672)
\moveto(375.020636876904088928234983598,585.396612813135554076512185302)\lineto(196.695091875634010095731241917,228.745522810595396411504701939)
\moveto(472.285978435618904584284663806,501.575300316964152144200227595)\lineto(198.424699683035847855799772405,227.714021564381095415715336194)
\moveto(370.710678118654752440084436211,587.082869338697069279187436616)\lineto(370.710678118654752440084436211,212.917130661302930720812563384)
\moveto(472.285978435618904584284663806,298.424699683035847855799772405)\lineto(198.424699683035847855799772405,572.285978435618904584284663806)
\moveto(487.082869338697069279187436616,329.289321881345247559915563789)\lineto(112.917130661302930720812563384,329.289321881345247559915563789)
\moveto(487.082869338697069279187436616,470.710678118654752440084436211)\lineto(112.917130661302930720812563384,470.710678118654752440084436211)
\moveto(401.575300316964152144200227595,227.714021564381095415715336194)\lineto(127.714021564381095415715336194,501.575300316964152144200227595)
\moveto(229.289321881345247559915563789,587.082869338697069279187436616)\lineto(229.289321881345247559915563789,212.917130661302930720812563384)
\moveto(401.575300316964152144200227595,572.285978435618904584284663806)\lineto(127.714021564381095415715336194,298.424699683035847855799772405)
\strokepath
\end{picture}
\end{center}
\caption{The crystallographic arrangement of rank three with $37$ hyperplanes.\label{cry37}}
\end{figure}

The proof of this classification relies on enumerations by the computer. In rank three, approximately $60.000.000$ cases need to be considered and $55$ such arrangements are found, the largest one having $37$ hyperplanes (see Figure \ref{cry37}).
A ``short'' proof would be a great improvement, even if the number of cases was only reduced to several thousands.
In this article, we prove the following corollary to the above theorem without the use of a computer:

\begin{theor}[Theorem \ref{mainthm}]\label{mainth}
Let $r>2$. Then there are only finitely many equivalence classes of irreducible crystallographic arrangements of rank $r$.
\end{theor}

In other words: we prove that the algorithms presented in \cite{p-CH09c} and \cite{p-CH10} terminate after finitely many steps without running them.
The proof of this finiteness in each rank greater than two relied on these computations in the original proof of Theorem \ref{cryarrclas}.

All theorems required for the original classification, as for instance in \cite{p-CH09c}, were formulated and proved in the terminology of Weyl groupoids. This is reasonable since the Weyl groupoid is the structure which appears naturally in the theory of Nichols algebras. However, the axioms of Weyl groupoids and the involved technical details tend to discourage the reader.
This is why we give new proofs for some of the ``old'' theorems in the terminology of arrangements and avoid groupoids almost entirely.
Moreover, this allows us to develop further notions and results on simplicial arrangements in general as for example Lemma \ref{lem1} or Definition \ref{refgroupoid} (cf.\ \cite{CMW} and \cite{CM17}).

This article includes all proofs required for our main result (and hence recovers some of the known results), except the proofs of Theorems \ref{thmranktwo} and \ref{sumofroots} (which are \cite[Prop.\ 3.7]{p-CH09d} and \cite[Thm.\ 2.10]{p-CH09c}).
Section \ref{sec:refcart} is devoted to arbitrary simplicial arrangements, whereas in Section \ref{sec:cryarr}, we recall the definitions of crystallographic arrangements and roots. Section \ref{sec:local} is about the structure of localizations and Section \ref{sec:bounds} contains the proof of Theorem \ref{mainth}.

\medskip
\noindent{\bf Acknowledgement:}
{I would like to thank C.~Bessenrodt and T.~Holm for very helpful discussions.}

\section{Reflections and Cartan matrices}\label{sec:refcart}

We first recall the notions of simplicial and crystallographic arrangements (compare \cite[1.2, 5.1]{OT} and \cite{p-C10}).

\begin{defin}\label{A_R_2}
Let $r\in\NN$, $V:=\RR^r$, and $(\Ac,V)$ be an \df{arrangement of hyperplanes} in $V$, that is, a finite set of linear hyperplanes $\Ac$ in $V$.
Let $\Kc(\Ac)$ be the set of connected components (\df{chambers}) of $V\backslash \bigcup_{H\in\Ac} H$.
If every chamber $K$ is an \df{open simplicial cone}, i.e.\ there exist
$\alpha^\vee_1,\ldots,\alpha^\vee_r \in V$ such that
\begin{equation*}
K = \Big\{ \sum_{i=1}^r a_i\alpha^\vee_i \mid a_i> 0 \quad\mbox{for all}\quad
i=1,\ldots,r \Big\} =: \langle\alpha^\vee_1,\ldots,\alpha^\vee_r\rangle_{>0},
\end{equation*}
then $\Ac$ is called a \df{simplicial arrangement}.
\end{defin}

\begin{examp}
Let $W$ be a real reflection group acting on $V$, $\RS\subseteq V^*$ the set of roots of $W$.
Then $\Ac = \{\ker \alpha \mid \alpha\in \RS\}$ is a simplicial arrangement.
The reflection arrangement is the most symmetric type of simplicial arrangement, one cannot ``distinguish'' the chambers, they all look the same.
\end{examp}

\begin{defin}[cf.\ {\cite[2.13, 2.15, 2.5]{OT}}]
The \df{product} $(\Ac_1 \times \Ac_2,V_1 \oplus V_2)$ of two arrangements $(\Ac_1,V_1)$, $(\Ac_2,V_2)$
is defined by
\begin{equation*}
\Ac_1 \times \Ac_2 = \{ H_1 \oplus V_2 \mid H_1 \in \Ac_1 \} \cup \{ V_1 \oplus H_2 \mid H_2 \in \Ac_2 \}.
\end{equation*}
If an arrangement $(\Ac,V)$ can be written as a non-trivial product $(\Ac,V) = (\Ac_1 \times \Ac_2,V_1\oplus V_2)$,
then $\Ac$ is called \df{reducible}, otherwise \df{irreducible}.\\
The \df{rank}\footnote{All arrangements $(\Ac,V)$ considered in this article satisfy $\rank \Ac=\dim V$.} of an arrangement $(\Ac,V)$ is $\rank \Ac := \dim (V) - \dim (\bigcap_{H\in \Ac} H)$.
\end{defin}

If $\Ac$ is simplicial then unlike for reflection groups, there are, in general, no linear maps acting as permutations on $\Ac$.
However, the base change maps between adjacent chambers are still reflections, i.e.\ linear automorphisms of finite order fixing a hyperplane.

\begin{defin}
Let $K$ be a field, $r\in\NN$, $V:=K^r$, and $H$ a hyperplane in $V$.
A \df{reflection} on $V$ at $H$ is a $\sigma\in\GL(V)$, $\sigma\ne\id$ of finite order which fixes $H$.
Notice that the eigenvalues of $\sigma$ are $1$ and $\zeta$ for some root of unity $\zeta\in K$.
\end{defin}

\begin{lemma}\label{lem1}
Let $\Ac$ be a simplicial arrangement and $K$ a chamber, i.e.\ there is a basis $B^\vee=\{\alpha^\vee_1,\ldots,\alpha^\vee_r\}$ of $V$ such that $K=\langle B^\vee\rangle_{>0}$. Let $\tilde K$ be the chamber with
\[ \overline{K}\cap\overline{\tilde K} = \langle \alpha^\vee_2,\ldots,\alpha^\vee_r\rangle_{\ge 0}. \]
Then there is a unique $\beta^\vee\in V$ with
\[ \tilde K=\langle \tilde B^\vee \rangle_{>0}, \quad
\tilde B^\vee = \{\beta^\vee,\alpha^\vee_2,\ldots,\alpha^\vee_r\}, \quad
\text{and}\quad |B\cap - \tilde B|=1, \]
where $B:=(B^\vee)^*$ and  $\tilde B:=(\tilde B^\vee)^*$ denote the dual bases.
\end{lemma}

\begin{proof}
Choose $\beta^\vee\in V$ such that $\tilde K=\langle \beta^\vee,\alpha^\vee_2,\ldots,\alpha^\vee_r\rangle_{> 0}$.
Let $\mu_1,\ldots,\mu_r\in\RR$ be such that $\beta^\vee=\sum_{i=1}^r \mu_i \alpha^\vee_i$ (notice $\mu_1\ne 0$).
Let $\tilde B=\{\beta_1,\ldots,\beta_r\}$ be the dual basis of $\{\beta^\vee,\alpha^\vee_2,\ldots,\alpha^\vee_r\}$, and
$B=\{\alpha_1,\ldots,\alpha_r\}$ be dual to $B^\vee$.
Then $\beta_1=\frac{1}{\mu_1}\alpha_1$ and $\beta_j=-\frac{\mu_j}{\mu_1}\alpha_1+\alpha_j$ for $j>1$.
To obtain
$|B\cap - \tilde B|=1$
we need
$-\alpha_1=\beta_1\in \tilde B$
and hence $\mu_1=-1$, $\beta_1=-\alpha_1$ and $\beta_j=\mu_j\alpha_1+\alpha_j$ for $j>1$. Thus a $\beta^\vee$ as desired exists and is unique.
\end{proof}

\begin{corol}\label{corlem1}
Using the notation of the proof of Lemma \ref{lem1},
the map
\[ \sigma : V^* \rightarrow V^*, \quad \alpha_i \mapsto \beta_i \]
is a reflection. With respect to $B=(B^\vee)^*$, it becomes the matrix
\[
\begin{pmatrix}
        -1 & \mu_2 & \dots & \mu_r \\
        0 &1&& 0\\
        \vdots & &\ddots&\\
        0 & 0 && 1
\end{pmatrix}.
\]
\end{corol}

\begin{examp}\label{ex12}
Let $R=\{(1,0),(0,1),(1,2)\}\in (\RR^2)^*$, $\Ac=\{\alpha^\perp\mid \alpha\in R\}$.
Then $K=\langle B^\vee\rangle_{>0}$ is a chamber if $B^\vee=\{\alpha^\vee_1=(1,0),\alpha^\vee_2=(0,1)\}$,
$K'=\langle \tilde B^\vee\rangle_{>0}$ with $\tilde B^\vee=\{\tilde\beta^\vee=(-2,1),\alpha^\vee_2=(0,1)\}$ is an adjacent chamber.
To obtain $\mu_1=-1$, we need to choose $\beta^\vee=(-1,\frac{1}{2})$, hence $\mu_2=\frac{1}{2}$.
The unique reflection $\sigma$ is
\[ \begin{pmatrix} -1 & \frac{1}{2} \\ 0 & 1 \end{pmatrix} \]
with respect to $B=(B^\vee)^*$.
\end{examp}

\begin{defin}\label{def:cartan}
Let $\Ac$ be a simplicial arrangement, $K=\langle B^\vee \rangle_{>0}$, $B^\vee=\{\alpha^\vee_1,\ldots,\alpha^\vee_r\}$ a chamber,
and $B=\{\alpha_1,\ldots,\alpha_r\}$ be dual to $B^\vee$.
Then by Corollary \ref{corlem1}, there are reflections $\sigma_1,\ldots,\sigma_r$, represented by
\[
\begin{pmatrix}
        1 &  && & 0 \\
         & \ddots&\\
        \mu_{i,1} & \cdots & -1 & \cdots & \mu_{i,r} \\
         & & &\ddots&\\
        0 & & && 1
\end{pmatrix},
\]
for certain $\mu_{i,j}\in \RR$, $i\ne j$
with respect to $B$ and uniquely determined by $K$, $B$ and its adjacent chambers.

The matrix $C^{K,B} = (c_{i,j})_{1\le i,j \le r}$ with
\[ c_{i,j} := \begin{cases} -\mu_{i,j} & \text{if } i\ne j \\
2 & \text{if } i=j \end{cases} \]
is called the \df{Cartan matrix} of $(K,B)$ in $\Ac$.
Note that
\[ \sigma_i(\alpha_j) = \alpha_j - c_{i,j} \alpha_i \]
for all $1\le i,j\le r$.
We sometimes write $\sigma_i^{K,B}$ to emphasize that $\sigma_i$ depends on $K$ and $B$.
\end{defin}

\begin{examp}
\begin{enumerate}
\item Let $\Ac$ be as in Example \ref{ex12}. Then the Cartan matrix of $(K,B)$ is
\[ C^{K,B} = \begin{pmatrix} 2 & -\frac{1}{2} \\ -2 & 2 \end{pmatrix}. \]
\item If $W$ is a Weyl group with root system $\RS$, then all Cartan matrices of $(K,B)$ when $B$ is a set of simple roots for the chamber $K$ are equal and coincide with the classical Cartan matrix of $W$.
\end{enumerate}
\end{examp}

\begin{defin}\label{refgroupoid}
Let $\Ac$ be a simplicial arrangement in $V=\RR^r$.
We construct a category $\Cc(\Ac)$ with
\begin{itemize}
\item objects: $\Obj(\Cc(\Ac)) = \{ B=(\alpha_1,\ldots,\alpha_r)\in (V^*)^r \mid \langle B^* \rangle_{>0}\in\Kc(\Ac) \}$ (where the bases $B$ are ordered).
\item morphisms: for each $B=(\alpha_1,\ldots,\alpha_r)\in \Obj(\Cc(\Ac))$ and $i=1,\ldots,r$ there is a morphism
$\sigma_i^{K,B}\in \Mor(B,(\sigma_i^{K,B}(\alpha_1),\ldots,\sigma_i^{K,B}(\alpha_r)))$. All other morphisms are compositions of the generators $\sigma_i^{K,B}$.
\end{itemize}
A \df{reflection groupoid} $\Wc(\Ac)$ of $\Ac$ is a connected component of $\Cc(\Ac)$.
A \df{Weyl groupoid}\footnote{This is not the general definition of a Weyl groupoid. For a complete set of axioms, see \cite{p-CH09a}.} is a reflection groupoid for which all Cartan matrices are integral.
\end{defin}

Using the so-called gate property, one can prove the existence of a type function for the chamber complex of a simplicial arrangement. In other words:

\begin{propo}[cf.\ {\cite[Prop.\ 3.26, Lemma 3.29]{CMW}}]\label{typefunction}
Let $\Ac$ be a simplicial arrangement, $\Wc(\Ac)$ a reflection groupoid, and $B_1=(\alpha_1,\ldots,\alpha_r)$, $B_2=(\beta_1,\ldots,\beta_r)$ two objects with $\langle B_1^*\rangle_{>0} = \langle B_2^*\rangle_{>0}$.
Then there exist $\lambda_1,\ldots,\lambda_r$ such that $\alpha_i=\lambda_i \beta_i$ for all $i=1,\ldots,r$.\\
In particular, for a fixed reflection groupoid we obtain a unique labelling of the walls of each chamber with the labels $1,\ldots,r$.
\end{propo}

\begin{defin}\label{def:rho}
Let $\Ac$ be a simplicial arrangement, $\Wc(\Ac)$ a reflection groupoid, and $K=\langle B^* \rangle_{>0}$ a chamber for $B=(\alpha_1,\ldots,\alpha_r)\in \Obj(\Cc(\Ac))$.
For $i\in \{1,\ldots,r\}$, let $\rho_i(K)$ be the chamber adjacent to $K$ with common wall $\ker \alpha_i$.
We thus obtain well defined maps
\[ \rho_i : \Kc(\Ac) \mapsto \Kc(\Ac) \]
which satisfy $\rho_i^2 = \id$ by Proposition \ref{typefunction}.
\end{defin}

\section{Crystallographic arrangements}\label{sec:cryarr}

\begin{defin}[{\cite[Definition 2.3]{p-C10}}]\label{cryst:arr}
Let $\Ac$ be a simplicial arrangement in $V$ and $\RS\subseteq V^*$ a finite set
such that $\Ac = \{ \ker \alpha \mid \alpha \in \RS\}$ and $\RR\alpha\cap \RS=\{\pm \alpha\}$
for all $\alpha \in \RS$.
We call $(\Ac,V,\RS)$ a \df{crystallographic arrangement} if
for all chambers $K\in\Kc(\Ac)$:
\begin{equation}\label{equ:crys}
\RS \subseteq \sum_{\alpha \in B^K} \ZZ \alpha,
\end{equation}
where
\[ B^K = \{ \alpha\in \RS \mid \forall x\in K\::\:\alpha(x)\ge 0,\:\: \langle \ker \alpha \cap \overline{K}\rangle = \ker \alpha \} \]
corresponds to the set of walls of $K$.
\\
Two crystallographic arrangements $(\Ac,V,\RS)$, $(\Ac',V,\RS')$ in $V$ are called \df{equivalent}
if there exists $\psi\in\Aut(V^*)$ with $\psi(\RS)=\RS'$. We then write $(\Ac,V,\RS)\cong(\Ac',V,\RS')$.
\\
If $\Ac$ is an arrangement in $V$ for which a set $\RS\subseteq V^*$ exists such that $(\Ac,V,\RS)$ is crystallographic, then we say that $\Ac$ is \df{crystallographic}.
\end{defin}

\begin{examp}
\begin{enumerate}
\item Let $\RS$ be the set of roots of the root system
of a crystallographic reflection group (i.e.\ a Weyl group). Then
$(\{\ker \alpha \mid \alpha\in \RS\},V,\RS)$ is a crystallographic arrangement.
\item If $R_+ := \{(1,0),(3,1),(2,1),(5,3),(3,2),(1,1),(0,1)\}$,
then $(\{\alpha^\perp \mid \alpha\in R_+\},$ $\RR^2,R_+\cup -R_+)$ is a crystallographic arrangement.
\end{enumerate}
\end{examp}

If $\Ac$ is crystallographic then there is a reflection groupoid $\Wc(\Ac)$ as in Definition \ref{refgroupoid} which is a Weyl groupoid and the notion of root system of a Weyl group may be adapted:

\begin{defin}\label{rootscartan}
Let $(\Ac,V,\RS)$ be a crystallographic arrangement and $K$ a chamber.
Fixing an ordering for $B^K$, we obtain a unique reflection groupoid $\Wc(\Ac)$ and thus unique orderings for all $B^{K'}$, $K'\in \Kc(\Ac)$ by Proposition \ref{typefunction}.
Notice further that the crystallographic property (\ref{equ:crys}) implies that $\Wc(\Ac)$ is a Weyl groupoid and that there is a unique object in $\Wc(\Ac)$ for each chamber of $\Ac$.
Hence we obtain a unique coordinate map
$$ \coord^K : V\rightarrow \RR^r \quad \text{with respect to}\quad B^K.$$
The elements of the standard basis $\{\alpha_1,\ldots,\alpha_r\}=\coord^K(B^K)$ are called \df{simple roots}.
The set
\[ R^K := \{ \coord^K(\alpha) \mid \alpha\in \RS \} \subseteq \NN_0^r \cup -\NN_0^r \]
is called the set of \df{roots} of $\Ac$ at $K$. The roots in $R^K_+:=R^K\cap \NN_0^r$ are called \df{positive}.
Let $1\le i,j \le r$. Then it is easy to see that
\[
c^K_{i,j} = \begin{cases}
-\max\{ k\in\NN_{\ge 0} \mid k\alpha_i+\alpha_j \in R^K \} & i\ne j \\
2 & i=j
\end{cases},
\]
where $C^K:=(c^K_{i,j})_{i,j}$ is the Cartan matrix of $(K,B^K)$ as defined in Definition \ref{def:cartan}.
Recall that for every $i=1,\ldots, r$, we have a reflection $\sigma^K_i : \ZZ^r\rightarrow\ZZ^r$ defined by
$\sigma^K_i(\alpha_j) = \alpha_j - c^K_{i,j} \alpha_i$
for all $1\le j\le r$.\\
Remark that if $\tilde K$ is the chamber adjacent to $K$ with
\[ \langle \overline{K}\cap\overline{\tilde K}\rangle=\ker \alpha \quad \text{for}\quad \alpha\in R \quad \text{with}
\quad \coord^K(\alpha)=\coord^{\tilde K}(\alpha)=\alpha_i, \]
then Lemma \ref{lem1} or \cite[Lemma 2.9]{p-C10} imply $\sigma^K_i = \coord^{\tilde K} \circ (\coord^K)^{-1}$ and thus $\sigma^K_i(R^K)=R^{\tilde K}$.
Finally, remember that we have maps $\rho_i : \Kc(\Ac) \mapsto \Kc(\Ac)$, $K \mapsto \tilde K$ (Definition \ref{def:rho}).
\end{defin}

To avoid confusion, we use different fonts for the ``global'' set $\RS$ and the ``local'' representations $R^K$.
These local representations ``are'' the objects of the Weyl groupoid. Notice that in the crystallographic case we have
\[ \Mor(B^K,B^{\tilde K}) = \{ w^{K,\tilde K} := \coord^{\tilde K} \circ (\coord^K)^{-1} \} \]
for chambers $K$ and $\tilde K$.

A crystallographic arrangement ``lives'' in a lattice; the crystallographic property (\ref{equ:crys}) implies a certain ``saturation'' quantified by
the following ``volume'' function which will play an important role.

\begin{defin}\label{volume}
Let $m\in \NN$. By the Smith normal form there is a unique left
$\GL (\ZZ^r)$-invariant right $\GL (\ZZ^m)$-invariant function
$\Vol_m : (\ZZ^r)^m\to \ZZ$ such that
\begin{align}
  \Vol_m(a_1\al _1,\dots,a_m\al _m)=|a_1\cdots a_m| \quad
  \text{for all $a_1,\dots,a_m\in \ZZ $,}
\end{align}
where $|\cdot |$ denotes absolute value,
i.e.\ $\Vol_m(\beta_1,\ldots,\beta_m)$ is the product of the elementary divisors of the matrix with columns $\beta _1,\dots,\beta _m$.
\end{defin}

In particular, if $m=1$ and $\beta \in \ZZ^r\setminus \{0\}$, then $\Vol_1(\beta )$ is the greatest common divisor of the coordinates of $\beta$. Further, if $m=r$ and $\beta _1,\dots,\beta _r\in \ZZ^r$,
then $\Vol_r(\beta _1,\dots,\beta _r)$ is the absolute
value of the determinant of the matrix with columns $\beta _1,\dots,\beta _r$.

Definition \ref{volume} yields a ``\df{volume}'' for roots:

\begin{defin}
Let $(\Ac,V,\RS)$ be an irreducible crystallographic arrangement of rank $r$. By the crystallographic property (\ref{equ:crys}), for chambers $K$, $K'$, the bases $B^K$ and $B^{K'}$ differ by a map in $\GL (\ZZ^r)$. Thus for $\beta_1,\ldots,\beta_m\in \RS$,
\[ \Vol_m(\coord^K(\beta_1),\ldots,\coord^K(\beta_m)) = \Vol_m(\coord^{K'}(\beta_1),\ldots,\coord^{K'}(\beta_m)). \]
Hence we have a well-defined map
\[ \Vol_m : \RS^m \rightarrow \ZZ, \quad (\beta_1,\ldots,\beta_m) \mapsto \Vol_m(\coord^K(\beta_1),\ldots,\coord^K(\beta_m)) \]
which does not depend on the choice of $K$.
\end{defin}

\section{Localizations}\label{sec:local}

\subsection{Localizations in root systems}

\begin{defin}
Let $\Ac$ be an arrangement and $X\le V$.
Let $\tilde\Ac_X := \{ H \in \Ac \mid X \subseteq H \}$, $U:=\bigcap_{H\in \tilde\Ac_X} H$, and $\pi : V\rightarrow V/U$ be the canonical projection.
The \df{localization} of $\Ac$ at $X$ is the arrangement
$(\Ac_X := \{ \pi(H) \mid H \in \tilde\Ac_X \}, V/U)$.
\\
The \df{intersection lattice} $L(\Ac)$ of $\Ac$ is the set of all finite intersections $H_1 \cap \ldots \cap H_k$ with $H_1,\ldots,H_k\in \Ac$.
\end{defin}

\begin{remar}
If $\Ac$ is simplicial (resp.\ crystallographic), then all localizations
are simplicial (resp.\ crystallographic) (cf.\ \cite{p-CRT11} or \cite{CMW}).
\end{remar}

\begin{remar}\label{locroots}
It is easy to understand localizations using roots:
Let $(\Ac,V,\RS)$ be a crystallographic arrangement, $X\le V$, without loss of generality $X\in L(\Ac)$,
and let $\pi : V\rightarrow V/X$ be the canonical projection.
Consider $\Ss:=\{\alpha\in \RS\mid X\subseteq \ker(\alpha)\}$.
Any $\alpha\in \Ss$ defines a linear form
$$ \iota(\alpha) : V/X \rightarrow \RR,\quad v+X\mapsto \alpha(v)$$
in $(V/X)^*$; remark that $\iota : \Ss\rightarrow (V/X)^*$ is injective.
Then $(\Ac_X,V/X,\iota(\Ss))$ is the localization and it is a crystallographic arrangement.
Moreover, $\Ss = \RS\cap \langle \Ss \rangle_\RR$.
Thus localizations correspond to subsets of $\RS$ of all roots contained in a fixed subspace\footnote{Notice that localizations are related to parabolic subgroupoids of the Weyl groupoid associated to $\Ac$; we omit these notions because we will not need them.}.
The set of chambers of $\Ac_X$ is
\[ \{ \pi(K) \mid K \text{ chamber of } \Ac, \:\: \overline{K}\cap X \ne 0\}. \]
\end{remar}

\begin{defin}
Let $(\Ac,V,\RS)$ be a crystallographic arrangement and $K$ a chamber.
For a subspace $X\le \RR^r$, we call $S_{K,X}:=X \cap R^K$ a \df{localization} of the crystallographic arrangement at $K$ and $X$. Notice that
\[ S_{K,X} = {S_{K,X}}_+ \dot\cup -{S_{K,X}}_+ \quad \text{for} \quad {S_{K,X}}_+:=X \cap R^K_+. \]
\end{defin}

\begin{lemma}\label{ssloc}
Let $(\Ac,V,\RS)$ be a crystallographic arrangement, $K$ a chamber, and $X\le \RR^r$.
Then there is a subset $\Delta\subseteq X\cap R^K_+$ which is a set of simple roots for the localization ${S_{K,X}}=X \cap R^K$, i.e.\
\[ {S_{K,X}}_+ \subseteq \sum_{\alpha\in \Delta} \NN_0 \alpha. \]
\end{lemma}
\begin{proof}
Let $\pi : V\rightarrow V/X$ be the canonical projection and $\iota$ the map defined in Remark \ref{locroots}.
Let $K'$ be the chamber in the localization such that $\pi(K)\subseteq K'$.
If $\Delta'$ is the set of simple roots of $\Ac_X$ at $K'$, then $\Delta:=\iota^{-1}(\Delta')\subseteq X\cap R^K_+$ is as desired.
\end{proof}

\subsection{Rank two}\label{secranktwo}

\begin{defin}\label{R_seq}
Define \df{$\cEs$-sequences} as finite sequences of length $\ge 2$
with entries in $\NN _0^2$ given by the following recursion.
\begin{enumerate}
\item\label{Fseq1} $((0,1),(1,0))$ is an $\cEs$-sequence.
\item\label{addtr} If $(v_1,\ldots,v_n)$ is an $\cEs$-sequence, then
$(v_1,\ldots,v_i,v_i+v_{i+1},v_{i+1},\ldots,v_n)$
are $\cEs$-sequences for $i=1,\ldots,n-1$.
\item Every $\cEs$-sequence is obtained recursively by 
\eqref{Fseq1} and \eqref{addtr}.
\end{enumerate}
\end{defin}

\begin{remar}
Notice that by definition, the construction of an $\cEs$-sequence $(v_1,\ldots,v_n)$ produces a triangulation of an $n$-gon by non-intersecting diagonals: step (\ref{Fseq1}) is the $2$-gon, including a sum $v_i+v_{i+1}$ as in (\ref{addtr}) corresponds to adding a triangle.
It is easy to check that an $\cEs$-sequence consists of vectors in a pair of diagonals in the Conway-Coxeter frieze pattern associated to this triangulation (cf.\ \cite{p-C14}).
\end{remar}

\begin{theor}[cf.\ {\cite{p-CH09b}}]\label{thmranktwo}
Let $(\Ac,V)$ be an arrangement of rank two and $\RS\subseteq V^*$ such that $\Ac = \{ \ker \alpha \mid \alpha \in \RS\}$ and $\RR\alpha\cap \RS=\{\pm \alpha\}$ for all $\alpha \in \RS$.
Then $(\Ac,V,\RS)$ is a crystallographic arrangement if and only if there exists a chamber $K$ such that $R^K_+$ is an $\cEs$-sequence.
In this case, $R^K_+$ is an $\cEs$-sequence for all chambers $K$.
\end{theor}

\begin{remar}\label{quid121}
A crystallographic arrangement $\Ac$ of rank two and a chamber $K$ define a sequence of negative Cartan entries
\[ (c_1,\ldots,c_n):=(-c_{1,2}^K,-c_{2,1}^{\rho_1(K)},-c_{1,2}^{\rho_2(\rho_1(K))},\ldots) \]
called a \df{quiddity cycle}, where $n=|\Ac|$.
Quiddity cycles are built like $\cEs$-sequences:
\begin{enumerate}
\item\label{qFseq1} $(0,0)$ is a quiddity cycle.
\item\label{qaddtr} If $(c_1,\ldots,c_n)$ is a quiddity cycle, then
$(c_1,\ldots,c_i+1,1,c_{i+1}+1,\ldots,c_n)$
are quiddity cycles for $i=1,\ldots,n-1$.
\item Every quiddity cycle is obtained recursively by 
\eqref{qFseq1} and \eqref{qaddtr}.
\end{enumerate}
Notice that this construction implies that the only quiddity cycles containing $(1,2,1)$ are $(1,2,1,2)$ and $(2,1,2,1)$
since the only quiddity cycle containing $(1,1)$ is $(1,1,1)$.
\end{remar}

From Theorem \ref{thmranktwo} we obtain in particular:

\begin{corol}\label{ranktwo}
Let $(\Ac,V,\RS)$ be a crystallographic arrangement of rank two and $K$ a chamber.
\begin{enumerate}
\item\label{sor2} Any $\alpha\in R^K_+$ is either simple or the sum of two positive roots in $R^K_+$.
\item\label{con2} If $\alpha$, $\beta$ are simple roots and $k\alpha+\beta \in R^K_+$, then $\ell\alpha+\beta\in R^K_+$ for all $\ell=0,\ldots,k$.
\end{enumerate}
\end{corol}

Corollary \ref{ranktwo} (\ref{sor2}) may be extended to arbitrary rank, we omit the proof because it involves the length function of a Weyl groupoid:

\begin{theor}[cf.\ {\cite[Thm.\ 2.10]{p-CH09c}}]\label{sumofroots}
Let $(\Ac,V,\RS)$ be a crystallographic arrangement, $K$ a chamber, and $\alpha\in R^K_+$ a positive root.
Then either $\alpha$ is simple, or it is the sum of two positive roots in $R^K_+$.
\end{theor}

Corollary \ref{ranktwo} (\ref{con2}) extends to arbitrary rank as well, see Lemma \ref{lemcon} below.

\subsection{Localizations in rank three}

In this section we assume that $r=3$, i.e.\ $V=\RR^3$.

\begin{lemma}[cf.\ {\cite[Prop.\,4.6]{p-CH09a}} or {\cite[Lem.\ 3.2]{p-C12}}]\label{adj2}
Let $(\Ac,V,\RS)$ be a crystallographic arrangement of rank three and $K$ a chamber.
Then $(\Ac,V)$ is reducible if $|R^K_+\cap \langle \alpha_1, \alpha_2\rangle|=|R^K_+\cap \langle \alpha_1, \alpha_3\rangle|=2$.
\end{lemma}
\begin{proof}
Since $\sigma_1^K(\alpha_2)=\alpha_2$, $\sigma_1^K(\alpha_3)=\alpha_3$, the chamber $\rho_1(K)$ is also adjacent to the localization $\langle \alpha_2,\alpha_3\rangle$.
But then any further $\beta\in R_+^K\backslash \{\alpha_1\}$ is in $\langle \alpha_2,\alpha_3\rangle$,
thus $\Ac$ is a so-called near pencil arrangement which is reducible.
\end{proof}

\begin{defin}\label{quiaux}
Let $(\Ac,V,\RS)$ be a crystallographic arrangement, $K_1$ a chamber, $1\le i\ne j\le r$, and $n:=|\langle\alpha_i,\alpha_j\rangle\cap R^K_+|$.
We denote the $2n$ chambers adjacent to the localization $\langle\alpha_i,\alpha_j\rangle$ by $K_1,\ldots,K_{2n}$:
for $\ell>1$, let
\[ K_\ell := \begin{cases}
\rho_i(K_{\ell-1}) & \text{if } \ell \text{ is even}, \\
\rho_j(K_{\ell-1}) & \text{if } \ell \text{ is odd}.
\end{cases}
\]
Notice that $K_{2n+1}=K_1$.
This sequence of chambers yields two sequences of integers:
\[ c_\ell := \begin{cases}
-c_{i,j}^{K_\ell} & \text{if } \ell \text{ is odd}, \\
-c_{j,i}^{K_\ell} & \text{if } \ell \text{ is even},
\end{cases}
\quad \quad
d_\ell := \begin{cases}
-c_{i,k}^{K_\ell} & \text{if } \ell \text{ is odd}, \\
-c_{j,k}^{K_\ell} & \text{if } \ell \text{ is even}
\end{cases}
\]
for $\ell=1,\ldots,2n$ and the unique $k\notin \{i,j\}$ with $1\le k\le r=3$.
We call $(c_1,\ldots,c_n)$ the \df{quiddity cycle} and $(d_1,\ldots,d_{2n})$ the \df{auxiliary cycle} of the localization $\langle \alpha_i,\alpha_j\rangle$.
\end{defin}

\begin{remar}
Let $(c_1,\ldots,c_{2n})$ be the above sequence. Then $(c_1,\ldots,c_n)=(c_{n+1},\ldots,c_{2n})$ and
$(c_1,\ldots,c_n)$ is the quiddity cycle of the frieze pattern associated to the localization at $\langle\alpha_i,\alpha_j\rangle$ as presented in Remark \ref{quid121}.
\end{remar}

\begin{figure}[ht]
\includegraphics[width=0.8\textwidth]{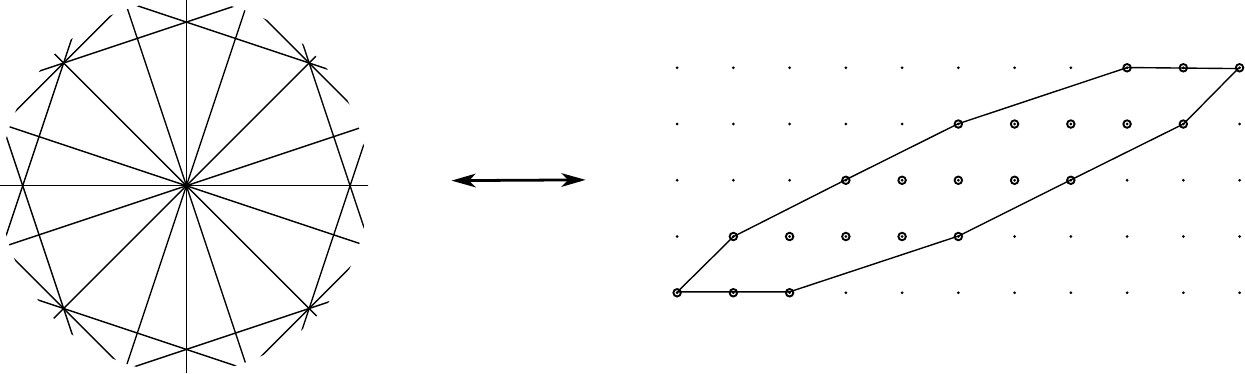}
\caption{A localization and the roots on the boundary in the dual space.\label{fig:local}}
\end{figure}

Figure \ref{fig:local} shows a part of a crystallographic arrangement of rank three on the left.
The right picture depicts the roots of the form $k\alpha_1+\ell\alpha_2+\alpha_3\in R^K_+$, i.e.\ all roots on the ``$(*,*,1)$-plane''.
Here, $\alpha_1$ and $\alpha_2$ are simple roots for a chamber of the localization of rank two with eight positive roots on the left.
The following proposition explains the ``hull'' of the convex set on the right:

\begin{propo}\label{plane001}
Let $(\Ac,V,\RS)$ be an irreducible crystallographic arrangement of rank three and $K$ a chamber.
Let $\beta_1=(0,1,0),\beta_2,\ldots,\beta_{n-1},\beta_n=(1,0,0)$ be the roots in the localization $\langle \alpha_1,\alpha_2\rangle$ ordered in such a way that $(\beta_1,\ldots,\beta_n)$ is an $\cEs$-sequence (ignoring the third coordinate $0$).
Let $(d_1,\ldots,d_{2n})$ be the auxiliary cycle of the localization $\langle \alpha_2,\alpha_1\rangle$.
\begin{enumerate}
\item Then \[ \gamma_\ell:=\alpha_3+\sum_{k=1}^\ell d_k \beta_k, \quad \delta_\ell:=\alpha_3+\sum_{k=1}^\ell d_{2n+1-k} \beta_{n+1-k}, \]
$\ell=0,\ldots,n$ are positive roots in $R^K$ with third coordinate $1$.
These are the vertices of the convex set in the $(*,*,1)$-plane.
\item\label{consec_d} There are no consecutive $d_\ell$'s both equal to $0$.
\item $|\{\gamma_\ell \mid \ell=0,\ldots,n\}|\ge n/2$ and $\gamma_{\ell+1}-\gamma_\ell\in \NN_0^3$.\label{plane001vert}
\end{enumerate}
\end{propo}
\begin{proof}
(1) Let $K_\ell$ and $c_\ell$ be as in Definition \ref{quiaux}.
Notice first that $\beta_\ell = w^{K_\ell,K}(\alpha_i)$ for some $i\in \{1,2\}$ depending on the parity of $\ell$.
Then for $\ell>1$,
\[
w^{K_\ell,K}(\alpha_3)
= w^{K_{\ell-1},K}(\sigma_i^{K_\ell}(\alpha_3))
= w^{K_{\ell-1},K}(\alpha_3 - c_{i,3} \alpha_i)
= w^{K_{\ell-1},K}(\alpha_3) + d_{\ell-1} \beta_{\ell-1}
\]
for some $i\in \{1,2\}$.
Since $w^{K_1,K}(\alpha_3)=\alpha_3$, the claim for the $\gamma_\ell$'s follows by induction.
The case of the $\delta_\ell$'s can be treated similarly.\\
(2) If $d_\ell=d_{\ell+1}=0$ for some $\ell$, then there exists a chamber $\tilde K$ with two localizations adjacent to $\tilde K$ with only $2$ positive roots; this is impossible by Lemma \ref{adj2} since the arrangement is simplicial and irreducible.\\
(3) This is by definition of $\gamma_\ell$ and by (2).
\end{proof}

The next lemma is a crucial tool. It extends the convexity which was observed in rank two to localizations and may be applied to pairs of roots in the $(*,*,1)$-plane:

\begin{lemma}[cf.\ {\cite[Lem.\ 3.4 and Lem.\ 3.2]{p-CH09c}}]\label{lemcon}
Let $(\Ac,V,\RS)$ be a crystallographic arrangement, $K$ a chamber, $k\in\NN_{\ge 2}$, $\alpha\in R^K_+$, $\beta\in \ZZ^r$, $\dim\langle \alpha,\beta\rangle_\QQ = 2$, $\alpha+k\beta\in R^K$, $\Vol_2(\alpha,\beta)=1$, and $(-\NN\alpha+\ZZ\beta)\cap \NN_0^r = \emptyset$.\\
Then $\beta\in R^K$ and $\alpha+\ell\beta\in R^K$ for all $\ell=0,\ldots,k$.
Moreover, there exists a chamber $K'$ and $1\le i,j \le r$ such that $-c^{K'}_{i,j}\ge k$.
\end{lemma}

\begin{proof}
We consider the localization $S_{K,X}:=X \cap R^K$ for $X:=\langle \alpha,\alpha+k\beta\rangle$.
Assume that $\beta \notin R^K$.
By Lemma~\ref{ssloc} there exist roots $\gamma _1,\gamma_2\in {S_{K,X}}_+$ such that ${S_{K,X}}_+\subseteq \NN_0\gamma_1+\NN_0\gamma_2$.
Since $\alpha,\beta\in X$ and $\Vol_2(\alpha,\beta)=1$, there exist
$m_1,\ell_1\in \ndZ$ and $m_2,\ell_2\in \ndZ $
such that $\gamma _1=m_1\al +m_2\beta $, $\gamma _2=\ell_1\al +\ell_2\beta $.
Notice that $m_1,\ell_1\ge 0$ because $(-\NN\alpha+\ZZ\beta)\cap \NN_0^r = \emptyset$;
since $\beta \notin R^a$, we even have that $m_1\ge 1$ and $\ell_1\ge 1$.

Now $\gamma_1,\gamma_2$ were chosen such that $\alpha\in \NN_0\gamma_1+\NN_0\gamma_2$, thus there exist $m,\ell\in\NN_0$ with $\alpha=m\gamma_1+\ell\gamma_2$,
hence $mm_1+\ell\ell_1=1$ and $mm_2+\ell\ell_2=0$ which implies $m=0$ or $\ell=0$.
If $k=0$, then $\ell=\ell_1=1$ and thus $\ell_2=0$ and $\gamma_2=\alpha$. If $\ell=0$, then $\gamma_1=\alpha$.

Assume without loss of generality that $\gamma_1=\alpha$. Then there exist $m',\ell'\in\NN_0$ such that $\alpha+k\beta=m'\alpha+\ell'\gamma_2$ and
hence $m'+\ell'\ell_1=1$, $k = \ell'\ell_2$. If $m'=1$, then $\ell'=0$ which contradicts $k\ge 2$. Thus $m'=0$, $\ell'=\ell_1=1$, and $k=\ell_2$; we get $\gamma_2=\alpha+k\beta$.
But this is a contradiction, since
$1=\Vol _2(\gamma _1,\gamma _2)=\Vol _2(\al ,\al +k\beta )=k>1$.
Thus $\beta \in R^K$.

For the remaining claims,
choose a chamber $K'$ adjacent to $\ker (\coord^K)^{-1}(\alpha) \cap \ker (\coord^K)^{-1}(\beta)$ such that $\beta':=w^{K,K'}(\beta)$ is simple and $\alpha':=w^{K,K'}(\alpha)\in R^{K'}_+$.
Then there is a $\delta\in R^{K'}_+$ such that $\beta',\delta$ are the simple roots of the localization $\langle \alpha',\beta'\rangle$,
and $\alpha' = k_1\beta'+k_2\delta$ for some $k_1,k_2\in \NN_0$. Now $k_2=1$ because $\Vol_2(\alpha',\beta')=1$.
But then $\delta$, $\alpha'+k\beta'=\delta+(k+k_1)\beta' \in R^{K'}_+$ and thus $\delta+\ell \beta'\in R^{K'}$ for $\ell=0,\ldots,k_1+k$ by Corollary \ref{ranktwo}.
In particular, $-c^{K'}_{i,j}\ge k+k_1$ if $\delta=\alpha_i$ and $\beta'=\alpha_j$ (see Definition \ref{rootscartan}).
\end{proof}

\begin{examp}
\begin{figure}[ht]
\includegraphics[width=0.4\textwidth]{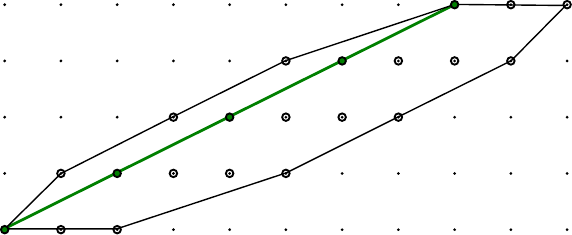}
\caption{Lemma \ref{lemcon} applied to the $(*,*,1)$-plane.\label{fig:lemcon}}
\end{figure}
With $\alpha=(0,0,1)$, $\beta=(2,1,0)$, and $k=4$, Lemma \ref{lemcon} implies the existence of the roots on the green line in Figure \ref{fig:lemcon}.
In fact, Lemma \ref{lemcon} implies that all lattice points in the convex set in Figure \ref{fig:lemcon} are roots.
\end{examp}

The next theorem is stronger than expected. If three roots have volume $1$, then they are close to be the walls of a chamber:

\begin{theor}[cf.\ {\cite[Thm.\ 3.10]{p-CH09c}}]\label{root_diffs}
Let $K$ be a chamber and $\al ,\beta,\gamma\in R^K_+$.
If $\Vol_3(\al ,\beta,\gamma)=1$
and none of $\alpha-\beta$, $\alpha-\gamma$, $\beta-\gamma$ are contained in
$R^K$, then $\al ,\beta,\gamma$ are the simple roots in $R^K$.
\end{theor}

We omit the proof since we do not need this result for our bounds, but notice that this Theorem~\ref{root_diffs} is the one in \cite{p-CH09c} with the most technical proof.
We obtain the following corollary which almost states that the $(*,*,1)$-plane is convex (see Example \ref{bluepath}):

\begin{corol}[cf.\ {\cite[Cor.\ 3.11]{p-CH09c}}]\label{convex_diff2}
Let $K$ be a chamber and $\gamma _1,\gamma _2,\al \in R^K$.
Assume that $\gamma_1,\gamma_2$ are simple roots and
that $\Vol _3(\gamma _1,\gamma _2,\al )=1$.
Then either
$\al$ is a simple root
or one of $\al -\gamma _1$, $\al-\gamma _2$ is contained in $R^K$.
\end{corol}

\begin{examp}\label{bluepath}
\begin{figure}[h]
\includegraphics[width=0.4\textwidth]{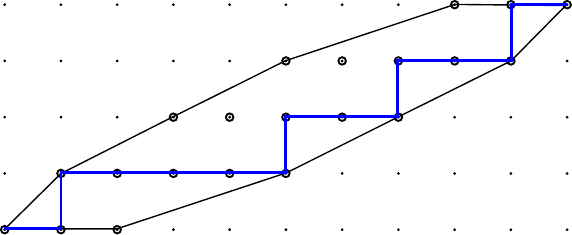}
\caption{A path of roots in the $(*,*,1)$-plane.\label{fig:path}}
\end{figure}
Repeatedly applying Corollary \ref{convex_diff2} with $\gamma_1=(1,0,0)$, $\gamma_2=(0,1,0)$, and starting with $\al =(10,4,1)$ yields (for example) the blue path of roots displayed in Figure \ref{fig:path}.
\end{examp}

\begin{remar}
A short proof for the fact that all lattice points in the convex hull of the roots in the $(*,*,1)$-plane are roots is still unknown.
\end{remar}

\begin{lemma}[cf.\ {\cite[Lem.\ 3.12 (2)]{p-CH09c}}]\label{r111}
Let $(\Ac,V,\RS)$ be an irreducible crystallographic arrangement of rank three and $K$ a chamber.
Then $\alpha_1+\alpha_2+\alpha_3\in R^K$.
\end{lemma}
\begin{proof}
Up to permuting coordinates, without loss of generality $|R^K_+\cap \langle \alpha_2, \alpha_3\rangle|>2$ by Lemma \ref{adj2}, hence $\al _2+\al _3\in R^K$.
Assume first that $|R^K_+\cap \langle \alpha_1, \alpha_3\rangle|=2$ which implies $c_{1,3}^K=c_{1,3}^{\rho_1(K)}=0$.
Then $c_{2,3}^{\rho_1(K)}<0$ and thus $\al _2+\al _3\in R^{\rfl _1(K)}_+$ and $\s^{\rfl _1(K)}_1(\al _2+\al _3)=-c^K_{12}\al _1+\al _2+\al _3\in
R^K$. Therefore $\alpha_1+\alpha_2+\alpha_3\in R^K$ holds by Lemma \ref{lemcon}
for $\al =\al _2+\al _3$ and $\beta =\al _1$.

Assume now that $c^K_{1,3}\not=0$. By symmetry and the previous paragraph we
may also assume that $c^K_{1,2},c^K_{2,3}\not=0$. Let $K'=\rfl _1(K)$.
If $c^K_{2,3}=0$ then $\al _1+\al _2+\al _3\in R^b$ by the previous
paragraph. Then
\[ R^K\ni \s^{K'}_1(\al _1+\al _2+\al _3)
    =(-c^K_{1,2}-c^K_{1,3}-1)\al _1+\al _2+\al _3,
\]
and the coefficient of $\al _1$ is positive.
Further, $\al _2+\al _3\in R^K$, and hence $\alpha_1+\alpha_2+\alpha_3\in R^K$ holds in this case by
Lemma \ref{lemcon}. Finally, if $c^{K'}_{2,3}\not=0$, then $\al _2+\al_3\in R^{K'}_+$, and hence $(-c^K_{1,2}-c^K_{1,3})\al _1+\al _2+\al _3\in R^K$.
Since $-c^K_{1,2}-c^K_{1,3}>0$, $\alpha_1+\alpha_2+\alpha_3\in R^K$ follows again from Lemma \ref{lemcon}.
\end{proof}

\section{Bounds}\label{sec:bounds}

Our first goal is to reproduce the bound for the Cartan entries (Theorem \ref{bound7}) as in \cite{p-CH09c}.
For this, we need the following technical result. We give a proof which is slightly different from the one in \cite{p-CH09c}, in particular, we leave no details to the reader:

\begin{theor}[cf.\ {\cite[Lem.\ 3.12 (5)]{p-CH09c}}]\label{k21}
Let $(\Ac,V,\RS)$ be a crystallographic arrangement of rank three, $K$ a chamber, and $|R^K_+\cap \langle \alpha_1, \alpha_2\rangle|\ge 5$.
Then
\[ k_0 := \min\{ k\in\NN_0 \mid k\alpha_1+2\alpha_2+\alpha_3\in R^K\} \in \{0,\ldots,4\} \]
and $k_0\le 2$ if $c^K_{1,3} =0$.
\end{theor}
\begin{proof}
Let $(c_1,\ldots,c_n)$ be the quiddity cycle, $(d_1,\ldots,d_{2n})$ the auxiliary cycle of $\langle\alpha_2,\alpha_1\rangle$, and $\gamma_0,\ldots,\gamma_n$ as in Proposition \ref{plane001}.
Then
\[
\gamma_0=(0,0,1), \quad
\gamma_1=(0,d_1,1), \quad
\gamma_2=(d_2,c_1 d_2+d_1,1), \]
\[
\gamma_3=(c_2 d_3+d_2, c_1c_2d_3 + c_1d_2 + d_1 - d_3,1), \]
\[
\gamma_4=(c_2c_3d_4 + c_2d_3 + d_2 - d_4, c_1c_2c_3d_4 + c_1c_2d_3 + c_1d_2 - c_1d_4 - c_3d_4 + d_1 - d_3,1),
\]
are positive roots.
Moreover, $(1,1,1)\in R^K$ by Lemma \ref{r111}.

Remark first that if $(0,c,1)\in R^K$ for $c>1$, then $(0,2,1)\in R^K$ by Lem.\ \ref{lemcon} since $\gamma_0=(0,0,1)\in R^K$.
Similarly, if $(1,c,1)\in R^K$ for $c>1$, then $(1,2,1)\in R^K$ by Lem.\ \ref{lemcon} since $(1,1,1)\in R^K$. Hence
\begin{equation}\label{0c11c1}
(k,c,1)\in R^K, \: k\le 1, \: c>1 \quad \Longrightarrow \quad k_0\le 1.
\end{equation}
Now we consider all possible values for the cycles.\\
If $d_1\ge 2$, then $k_0\le 1$ by (\ref{0c11c1}) since $\gamma_1\in R^K$.
Hence assume $d_1\le 1$.

We first consider the case $c_1>1$.\\
If $d_1=0$, then $d_2>0$ (Prop.\ \ref{plane001}, (\ref{consec_d})). Applying Lem.\ \ref{lemcon} to $\gamma_0,(d_2,c_1 d_2,1)=\gamma_2\in R^K$ gives $(1,c_1,1)\in R^K$, thus $k_0\le 1$ by (\ref{0c11c1}).\\
If $d_1=1$, $d_2>0$, then $\gamma_2 = d_2 (1, c_1,0) + \gamma_1$, thus $(1,c_1+1,1)\in R^K$ and $k_0\le 1$ by (\ref{0c11c1}).\\
If $d_1=1$, $d_2=0$, then $d_3>0$, $\gamma_3 = d_3(c_2, c_1c_2 - 1,0) + \gamma_1$ thus $(c_2,c_1 c_2,1)\in R^K$ which implies $(1,c_1,1)\in R^K$ and $k_0\le 1$ by (\ref{0c11c1}).

Now consider the case $c_1=1$. This implies $c_2>1$ since $|R^K_+\cap \langle \alpha_1, \alpha_2\rangle|\ge 5$.\\
If $d_1=1$, $d_2>0$, then $\gamma_2 = d_2 (1,1,0) + \gamma_1$, thus $(1,2,1)\in R^K$ and $k_0\le 1$.\\
If $d_1=1$, $d_2=0$, then $d_3>0$, $\gamma_3 = d_3(c_2, c_2 - 1,0) + \gamma_1$ thus $(c_2, c_2,1)\in R^K$ which implies $(2,2,1)\in R^K$ and $k_0\le 2$.\\
The last remaining case is $d_1=0$, and thus $d_2>0$. Notice that $d_1=0$ also implies $(1,0,1)\in R^K$ since $\delta_1 = (d_{2n},0,1)\in R^K$ and $d_{2n}>0$.
Recall also that we are still in the case $c_1=1$ and $c_2>1$.\\
If $d_2\ge 2$, then $\gamma_2 = (d_2,d_2,1)\in R^K$ and thus $(2,2,1)\in R^K$ and $k_0\le 2$.
Hence we may assume $d_2=1$.\\
If $d_3>0$ then $\gamma_3 = (c_2 d_3+1, c_2d_3 + 1 - d_3,1) = d_3(c_2,c_2-1,0)+(1,1,1)$, thus $(c_2+1,c_2,1)\in R^K$.
But $(c_2+1,c_2,1) = c_2(1,1,0)+(1,0,1)$ which implies $(3,2,1)\in R^K$ and $k_0\le 3$.\\
Finally, assume that $d_3=0$, $d_4>0$.
Then $\gamma_4=d_4(c_2c_3 - 1, c_2c_3 - 1 - c_3,0)+(1,1,1)$ implies
$(c_2c_3,c_2c_3-c_3,1) = c_3(c_2,c_2-1,0)+(0,0,1)\in R^K$.\\
If $c_2>2$, then $(c_2,c_2-1,1)=(c_2-1)(1,1,0)+(1,0,1)\in R^K$ and thus $(3,2,1)\in R^K$ and $k_0\le 3$.\\
If $c_2=2$, then $(2c_3,c_3,1)\in R^K$. If $c_3>1$ then this implies $(4,2,1)\in R^K$ and $k_0\le 4$.
The case $c_3=1$ is excluded since it implies $|R^K_+\cap \langle \alpha_1, \alpha_2\rangle|=4$: by Remark \ref{quid121}, the only quiddity cycles containing $(1,2,1)$ are $(1,2,1,2)$ and $(2,1,2,1)$. 

If $c^K_{1,3} =0$ then $d_{2n}=c^K_{1,3}=0$ implies $d_1>0$ by Prop.\ \ref{plane001}, (\ref{consec_d}). All above cases with positive $d_1$ imply $k_0\le 2$.
\end{proof}

This allows to compute a global bound for Cartan entries in crystallographic arrangements of rank greater than two:

\begin{theor}[cf.\ {\cite[Thm.\ 3.13]{p-CH09c}}]\label{bound7}
Let $(\Ac,V,\RS)$ be a crystallographic arrangement of rank greater or equal to three.
Then all entries of the Cartan matrices are greater or equal to $-7$.
\end{theor}
\begin{proof}
Assume that $K$ is a chamber with largest Cartan entry $-c_{1,2}^K\ge 8$.
Then Remark \ref{quid121} implies that $|R^K_+\cap \langle \alpha_1, \alpha_2\rangle|\ge 5$.
By Theorem \ref{k21} there exists $k_0\in \{0,1,2,3,4\}$ such that
$\gamma:=k_0\al _1+2\al _2+\al _3\in R^K_+$.
In the adjacent chamber $K'=\rho_1(K)$, we have
\[ \gamma':=\sigma_1^K(\alpha) = (-k_0 - 2 c_{1,2}^K - c_{1,3}^K) \alpha_1 + 2\alpha_2 + \alpha_3 \in R^{K'}_+. \]
Again by Theorem \ref{k21} there exists $k'_0\in \{0,1,2,3,4\}$ such that
$\alpha:=k'_0\al _1+2\al _2+\al _3\in R^{K'}_+$.
Now applying Lemma \ref{lemcon} to $\alpha$ and
$\gamma' = \alpha + (-k_0 - 2 c_{1,2}^K - c_{1,3}^K - k'_0) \alpha_1$
yields a chamber $K''$ with $1\le i,j \le 3$ and
\[ -c_{i,j}^{K''} \ge -k_0 - 2 c_{1,2}^K - c_{1,3}^K - k'_0. \]
By Theorem \ref{k21},
\[ k_0 \le \begin{cases}
2 & \text{if } -c_{1,3}^K = 0, \\
4 & \text{if } -c_{1,3}^K > 0,
\end{cases} \]
thus
\[ -c_{i,j}^{K''}  \ge \begin{cases}
-c_{1,2}^K + 2> -c_{1,2}^K & \text{if } -c_{1,3}^K=0, \\
-c_{1,2}^K -c_{1,3}^K > -c_{1,2}^K & \text{if } -c_{1,3}^K>0.
\end{cases} \]
This is a contradiction to the assumption that $-c_{1,2}^K$ is the largest Cartan entry.
\end{proof}
\begin{remar}
The classification of crystallographic arrangements shows that in fact, entries of the Cartan matrices are always greater or equal to $-6$.
\end{remar}

Notice that there are infinitely many non-equivalent crystallographic arrangements of rank two with Cartan entries greater or equal to $-7$.
However, the number of non-equivalent localizations of rank two in rank three is finite (Corollary \ref{finr2}). We first prove:

\begin{theor}\label{b128}
Any localization of rank two of an irreducible crystallographic arrangement of rank three has at most $128$ positive roots.
\end{theor}

\begin{proof}
Without loss of generality, assume that $|R_+^K\cap \langle \alpha_1,\alpha_2\rangle|>128$ for some chamber $K$.
Then by Proposition \ref{plane001}, (\ref{plane001vert}) there are more than $64$ roots of the form $k\alpha_1+\ell\alpha_2+\alpha_3$, i.e.\ there exist roots
$(a,b,1),(a',b',1)\in R^K$, $(a,b,1)\ne(a',b',1)$ with
\[ a \equiv a' \:\:(\md 8), \quad b \equiv b' \:\:(\md 8), \]
and by Proposition \ref{plane001}, (\ref{plane001vert}) we may assume $a\ge a'$ and $b\ge b'$. But then
\[ (a,b,1) = (a',b',1) + k ((a-a')/k,(b-b')/k,0) \]
for some $k\ge 8$ and $(a-a')/k,(b-b')/k\in \ZZ$, $\gcd((a-a')/k,(b-b')/k)=1$.
By Lemma \ref{lemcon}, this implies the existence of a Cartan entry less or equal to $-8$, contradicting Theorem \ref{bound7}.
\end{proof}

\begin{corol}\label{finr2}
There is a finite set $\Ic$ of equivalence classes of crystallographic arrangements of rank two such that
every localization of rank two of an irreducible crystallographic arrangement of rank three belongs to one of the classes in $\Ic$.
\end{corol}

\begin{proof}
By Theorem \ref{b128}, a localization of rank two of a crystallographic arrangement of rank three has at most $128$ positive roots.
Since a crystallographic arrangement $(\Ac,V,\RS)$ of rank two corresponds to a triangulation of a convex $|\RS|/2$-gon by non-intersecting diagonals
(see Section \ref{secranktwo}), there are only finitely many non-equivalent such arrangements with at most $128$ positive roots.
\end{proof}

\begin{corol}\label{boundvolm}
There exists a bound $m$, such that for any irreducible crystallographic arrangement of rank $r>2$ and $\alpha,\beta\in \RS$,
$$\Vol_2(\alpha,\beta)\le m.$$
\end{corol}

\begin{proof}
Viewing $\alpha$ and $\beta$ as elements of the localization $\langle \alpha,\beta\rangle$, we may choose a chamber $K$ such that
$\coord^K(\alpha)=\alpha_i$, $\coord^K(\beta)=a\alpha_i+b\alpha_j$ for suitable $a,b\in \ZZ$, without loss of generality $i=1$, $j=2$.
Since $r>2$, the roots $\coord^K(\alpha),\coord^K(\beta)$ are roots in a localization $\langle \alpha_1,\alpha_2,\alpha_\ell\rangle$ of rank three, $\ell>2$.
Thus by Corollary \ref{finr2}, the localization $\langle \alpha,\beta\rangle$ is one of finitely many possible crystallographic arrangements of rank two up to equivalence, hence by Section \ref{secranktwo}, coordinates of roots in these crystallographic arrangements are bounded by some number $m\in \NN$.
This implies
$\Vol_2(\alpha,\beta)=|b|\le m$.
\end{proof}

\begin{remar}
In fact, the classification of crystallographic arrangements shows that the sharp bound in Corollary \ref{boundvolm} is $m=6$.
\end{remar}

\begin{theor}\label{mainthm}
Let $r>2$. Then there are only finitely many equivalence classes of irreducible crystallographic arrangements of rank $r$.
\end{theor}

\begin{proof}
Let $K$ be a chamber of an irreducible crystallographic arrangement of rank $r>2$. Consider the map
\[ \psi : R^K_+ \rightarrow (\ZZ/(m+1)\ZZ)^r, \quad (a_1,\ldots,a_r) \mapsto (\overline{a_1},\ldots,\overline{a_r}). \]
Assume that $|R^K_+|>(m+1)^r$. Then there exist $\alpha,\beta\in R^K_+$, $\alpha\ne\beta$ and $\psi(\alpha)=\psi(\beta)$.
The volume $\Vol_2(\alpha,\beta)$ is divisible by $(m+1)$. Since $\alpha\ne\beta$, this contradicts Corollary~\ref{boundvolm}.\\
Hence there is a global bound for the number of positive roots. But the number of equivalence classes of irreducible crystallographic arrangements with bounded number of roots is bounded by Theorem \ref{sumofroots}.
\end{proof}

\def\cprime{$'$}
\providecommand{\bysame}{\leavevmode\hbox to3em{\hrulefill}\thinspace}
\providecommand{\MR}{\relax\ifhmode\unskip\space\fi MR }
\providecommand{\MRhref}[2]{%
  \href{http://www.ams.org/mathscinet-getitem?mr=#1}{#2}
}
\providecommand{\href}[2]{#2}

\end{document}